\documentclass[12pt]{amsart}

\usepackage{amsthm,amsmath,epsf,amssymb,xy}

\usepackage{graphicx}

\xyoption{all}

\def\dual                 {{\vee}} 
\def\rk                 {{\rm rk}}

\def\SSigma         {{\bf \Sigma}}

\def\ZZ                 {{\mathbb Z}} 
\def\PP                {{\mathbb P}} 
\def\RR                 {{\mathbb R}} 
\def\CC                 {{\mathbb C}} 
\def\QQ                 {{\mathbb Q}}

\def\cA    {{\mathcal A}}

\def\Ca    {{\mathcal C}} 
 
\def\Fa    {{\mathcal F}} 
\def\Ia    {{\mathcal I}}

\def\Oa    {{\mathcal O}} 
\def\Pa    {{\mathcal P}}

\def\Ta    {{\mathcal T}}

\newtheorem{lemma}{Lemma}[section] 
\newtheorem{theorem}[lemma]{Theorem} 
\newtheorem{corollary}[lemma]{Corollary} 
\newtheorem{proposition}[lemma]{Proposition} 
\newtheorem{definition}[lemma]{Definition} 
\newtheorem{conjecture}[lemma]{Conjecture} 
 
\theoremstyle{remark} 
\newtheorem{remark}[lemma]{Remark} 
\newtheorem*{proof*}{Proof} 
\newtheorem{example}[lemma]{Example}

\title[Discriminants and toric $K$--theory]{
Discriminants and toric $K$--theory}
 
\author{R. Paul Horja and Ludmil Katzarkov} 
\address{Department of Mathematics, University of Miami, Coral Gables, FL 33146, USA;
{\tt horja@math.miami.edu}} 
\address{Department of Mathematics, University of Miami, Coral Gables, FL 33146, USA;
NRU Higher School of Economics, Moscow, Russia; 
Institute of Mathematics and Informatics, Bulgarian Academy of Sciences, 
Acad. G. Bonchev Str. bl. 8, 1113, Sofia, Bulgaria;
\newline{\tt l.katzarkov@miami.edu}} 

\thanks{}

\begin{document} 

\begin{abstract} 
We discuss a categorical approach
to the theory of discriminants in the
combinatorial language 
introduced by Gelfand, Kapranov and 
Zelevinsky. Our point of view is 
inspired by homological mirror symmetry and 
provides $K$--theoretic evidence for  
a conjecture presented by Paul 
Aspinwall in a conference talk in Banff in March 2016
and later in a joint paper with Plesser and Wang. 
\end{abstract}

\maketitle

\medskip

\section{Introduction.}\label{sec0}
In this note, we investigate a conjecture
of Aspinwall, Plesser
and Wang \cite{APW}.
Our calculations offer supporting evidence at the 
level of toric $K$--theory. We enhance the
conjecture and propose a novel point of view on the multiplicities of 
discriminants which play an important implicit role in understanding 
the categorical aspects of the underlying toric birational geometry
and the related web of spherical functors. 

\

In the projective Calabi--Yau case, homological mirror symmetry \cite{kont1} is stated 
as an equivalence of the bounded
derived category 
of coherent sheaves $D^b(X)$ on a smooth projective Calabi--Yau variety $X$ and
the derived Fukaya category $DFuk(Y)$ of the mirror Calabi--Yau variety $Y.$ 
Mirror symmetry and string theory considerations about $D$--brane moduli spaces
also predict an identification between
the global structure of locally trivial families of categories $D^b(X)$ and $DFuk(Y)$
over the K\"ahler parameters of $X$ and
the complex parameters of $Y,$ respectively. 
As a by--product of his general conjecture, Kontsevich \cite{kont2} 
conjectured that the action on cohomology of the group of 
self--equivalences of $D^b(X)$ 
matches the monodromy action on the cohomology
of the mirror Calabi--Yau variety $Y.$  In particular, 
for a smooth projective Calabi-Yau variety $X,$ 
the action of of the spherical twist induced by the structure
sheaf of $X$ is ''mirrored" by the monodromy action 
around a distinguished ''primary" component in the 
moduli space of complex structures of $Y.$ While such a 
statement is implicit in the existing proofs of homological
mirror symmetry as well as in the statement of the Strominger--Yau--Zaslow 
conjecture \cite{SYZ} as identifications at large complex/radius limits,
our work is concerned with a similar global identification 
``far" away from such special points in the moduli spaces. 

\

In the present work, we consider 
the simplified case of a 
toric quasi-projective Calabi-Yau
Deligne-Mumford stack viewed as 
a resolution of an affine toric Gorenstein
singularity. This geometry is 
determined 
by a finite set of vectors $A \subset N=\ZZ^d$ contained
in an integral hyperplane at distance $1$ from the origin.  
The proposal by Aspinwall, Plesser and Wang expands
Kontsevich's
identification to all the irreducible components of the principal 
$A$-determinant $E_A.$ 
The polynomial $E_A$ 
has integer coefficients and 
generalizes the classical discriminant. It was introduced and studied by 
Gelfand, Kapranov and Zelevinsky \cite{GKZbook} and, among other 
remarkable features, has the
property that its Newton polytope coincides with the secondary polytope $S(A)$ 
which is combinatorially determined by the
starting configuration $A.$ One can view the 
main result of our work as a first step towards the categorification of 
the principal $A$-determinant $E_A.$ 

\

The APW conjecture predicts that, for each 
non-empty face $\Gamma$ 
of the polytope $Q= {\rm conv} (A),$ there exists 
a spherical functor $D_\Gamma \to D^b(X)$ corresponding 
to the component of the $A$-discriminant determined by $\Gamma.$ 
For any stacky fan $\SSigma$ supported 
on the cone over the polytope $Q=\rm{conv} (A),$ let 
$X=X_\SSigma$ denote the associated 
toric Deligne-Mumford stack as defined
by Borisov, Chen and Smith \cite{BCS}. 
In Section \ref{sec:k},
we give a 
conjectural construction of the category $D_\Gamma$ 
as the bounded derived category 
of coherent sheaves $D^b(X_{\SSigma_\Gamma})$ 
of the toric DM stack $X_{\SSigma_\Gamma}$. 
The stacky fan $\SSigma_\Gamma$ is determined by 
$\SSigma$ and $\Gamma$ (cf. Definition \ref{def:fangamma}) 
and the triangulated 
category $D_\Gamma$ is independent of the stacky fan 
$\SSigma,$ but its $t$--structure and its image in $D^b(X)$
may change under a toric crepant birational transformation 
induced by a generalized flop of the stacky fan $\SSigma.$
Intuitively, this change corresponds to some interesting
wall--crossing phenomena in the K\"ahler parameter moduli space
along the 
component of the $A$-discriminant determined by $\Gamma.$ 
Moreover, for any edge $F$ of the secondary polytope, there 
exists a toric Deligne--Mumford stack $Z_F$ 
and a ``wall--monodromy" spherical functor 
$D^b (Z_F) \to D^b (X)$ 
(cf. Definition \ref{def:fanz} and 
Proposition \ref{prop:ez}).
We state the main 
conjecture as follows. 

\medskip

{\bf  Conjecture \ref{conj:apw}.
}

{\it 
$1)$ For any edge $F$ of the secondary polytope, the category $D^b(Z_F)$
admits a semiorthogonal decompositon consisting of 
$n_{\Gamma,F}$ components $D^b (D_\Gamma)$ for each non-empty
face $\Gamma$
of the polytope $Q={\rm conv}(A),$ for some explicitily defined algebraic multiplicities 
$n_{\Gamma,F}$ (cf. Definition \ref{def:n}).

$2)$ For any non-empty face $\Gamma$ of the polytope $Q,$
there exists a spherical functor
$D^b(D_\Gamma) \to D^b (X)$ for any toric DM stack $X$ 
determined by a triangulation corresponding to a vertex of the secondary 
polytope. }

\

When $n_{\Gamma,F} > 0,$
the second part is a direct consequence of the first:
the wall monodromy functors $D^b (Z_F) \to D^b (X)$ 
are spherical, so an unpublished result of Kuznetsov,
and Halpern-Leistner--Shipman \cite{HLS}
implies that each component 
of the semiorthogonal decomposition determines 
a spherical functor. 

\

The main result of this work is the following theorem. 

\

{\bf Theorem  \ref{thmm}.
}
{\it 
For any edge $F$ of the secondary polytope, 
the following equality holds:
\begin{equation*}
\rk (K_0 (D^b(Z_F))) = \sum_{\Gamma \subset Q} n_{\Gamma,F} \cdot \rk (
K_0(D_\Gamma
)),
\end{equation*}
with the summation taken over all the non-empty 
faces $\Gamma$ of the polytope $Q.$ 
}

\

The theorem 
lends support to the conjecture by checking its consistency 
in terms of the ranks of the $K$-theory for the various toric Deligne-Mumford
stacks that enter the geometrical picture and local multiplicities of discriminants
and their intersecting components.

\

The importance of these multiplicities has been recognized in the 
work of Aspinwall, Plesser and Wang \cite{APW}, and their 
relevance is clearly explained 
in the more recent paper of Kite and Segal \cite{KS}. Such multiplicities are also
featured in earlier works \cite{AHK}, \cite{HLS}.
In the purely categorical sense, they are implicit in the theory of ''windows"
that began with the string theoretical work of Herbst, Hori and Page \cite{HHP},
and was made rigorous by Ballard, Favero, Katzarkov \cite{BFK} and Halpern-Leistner \cite{HL}. 
Our approach to determining multiplicities can be computationally involved but 
it is elementary algebraic 
(see Definition \ref{def:n}) and it uses the powerful GKZ toolbox 
for studying discriminants. Our definition avoids the well known subtleties related to
the local topological picture of intersecting discriminants. A
categorification of 
the classical braid factorization technique should offer 
a truly conceptual picture of the expected semiorthogonal decompositions
beyond the $K$--theoretic point of view. For a glimpse of how such a procedure
might work, see \cite{AHK}. 

\

In this note, we do not perform any mirror monodromy calculations
for the associated GKZ $D$--module. A true homological
mirror symmetry consistency check of the APW conjecture would require such computations,
but we leave this discussion for future work. However, 
by simply combining the results of this work with the 
analytic
monodromy calculations from \cite{hor}, \cite{BH2},
we can check the APW 
predictions in some cases under certain transversality assumptions, 
see Corollary \ref{cor:tr}. 
It is important to note that the principal $A$-determinant 
is determined by the characteristic cycle of the GKZ $D$--module.
We expect that the ''stacky" nature of the toric geometrical context
will require the use of the better behaved GKZ system 
\cite{BH} and an appropriately adapted version of the topological mirror map. 

\

{\it Acknowledgements.} We would like to thank Lev Borisov for answering some
questions about toric DM stacks. L. K. was partially supported by 
the Simons Investigators Award HMS,
Simons Collaborations Grant, NSF Grant,
National Science Fund of Bulgaria, National
Scientific Program-- Excellent Research and People for the Development
of European Science (VIHREN), Project No. KP-06-DV-7. The study has been funded within the 
framework of the HSE University Basic Research Program.

\

\section{Review of toric geometry and $A$-discriminants}\label{sec1}

\subsection{Toric Geometry}
We briefly review the definition of a  
toric Deligne-Mumford stack as developed in \cite{BCS}
(see also \cite{Jiang}). A \emph{stacky fan} ${\bf \SSigma}$ is defined as
the data $(\Sigma,\{v_i\})$ 
where $\Sigma$ is a simplicial fan in $\bar{N}= N \otimes_{\ZZ} \QQ,$ 
and $\{v_i\}$ $1\leq i \leq n,$
is a collection of elements in the 
finitely generated group $N.$ We assume that the rays of the 
simplicial fan $\Sigma$ are generated by 
the possibly non-minimal non-zero integral elements $\bar{v}_i.$

The set $\{v_i\}$ defines a map 
$
\alpha:\ZZ^n\to N
$
with finite cokernel. If we dualize, we get
an exact sequence
$$
0 \to N^\dual \to (\ZZ^n)^\dual \to N^{\prime} \to K \to 0.
$$
defining the Gale dual of $\alpha$ (see \cite{BCS}),
with $K$ a finite abelian group.
We apply the functor ${\rm Hom} (\cdot, \CC^{\times})$
to get the exact sequence
$$
0 \to K \to G \to (\CC^{\times})^n \to (\CC^{\times})^d \to 1,
$$
where $G$ is the algebraic group 
$
{\rm Hom}(N^{\prime},\CC^{\times}).
$

\medskip

Consider the subset $Z$ of $\CC^n$ that consists of all the
points ${\bf z}=(z_1,\ldots,z_n)$ such that the set of $v_i$ for
the zero coordinates of ${\bf z}$ is contained in a cone of $\Sigma$.
Then the toric DM stack $X_{\bf\SSigma}$
that corresponds to the stacky fan
${\SSigma}$ is defined as the stack quotient 
$[Z/G]$ where $Z$ and $G$ are 
endowed with the natural reduced 
scheme structures. 
It has been shown in \cite{BCS} that $X_{\SSigma}$
is a Deligne-Mumford stack whose moduli space is
the simplicial toric variety $X_\Sigma.$
The category of coherent sheaves on $[Z/G]$ 
is equivalent to that of $G$-linearized coherent sheaves on
$Z$, see \cite[Example 7.21]{Vistoli}. 

\medskip

To a cone $\sigma$ in $\Sigma,$ one can associate a closed substack 
$X_{\sigma}$ of 
$X_{\SSigma}$ by looking at the quotient of $N$ by the subgroup
$N_{\sigma}$ spanned by $v_i$ with $\bar{v}_i \in \sigma$. The cones of the 
new fan $\Sigma/\sigma$ are obtained as images in $N/N_{\sigma} \otimes \RR$
of the cones in the star
of $\sigma$ in $\Sigma$, consisting of those cones in $\Sigma$
containing $\sigma$ as a subcone. 
The elements $\bar{v}_i$ for the fan $\Sigma/\sigma$ 
are images of those $v_i$ that belong to cones in the star of $\sigma$
but are not in $\sigma.$
The quotient $N/N_{\sigma}$ may 
have torsion, even in the case of a torsionfree $N.$ Therefore,
the closed substack $X_{\sigma}$ may not be reduced even though
$X_{\SSigma}$ is.

\medskip

\subsection{Combinatorics of $A$-sets}\label{section:A}
In this section, we review some concepts discussed in chapter 7 of the
book \cite{GKZbook}.
A finite subset $A= \{ v_1, \ldots, v_n \}$ of 
the integral lattice $N=\ZZ^d$ is called an {\bf $A$--set} if
the set $A$ generates the lattice $N$ as an abelian group, and 
there exists a linear function 
$h : N \to \ZZ$ such that $h(v_i)=1,$ for all 
$i, 1 \leq i \leq n.$

For the rest of the paper, we will assume that the set 
$$A= \{ v_1, \ldots, v_n \} \subset N=\ZZ^d$$ 
is an $A$--set. However, see Remark \ref{rm1}
below. 

We will use the notations 
\begin{equation*}
\begin{split}
Q&:= {\rm conv} (A) \subset \RR^{d-1},\\
K&:= \sum_{1 \leq i \leq n} \RR_{\geq 0} v_i= \RR_{\geq 0} Q \subset N \otimes \RR=\RR^{d},\\
S&:= \sum_{1 \leq i \leq n} \ZZ_{\geq 0} v_i= \ZZ_{\geq 0} A\subset N=\ZZ^{d}.
\end{split}
\end{equation*}
The rational polyhedral cone $K$ 
defines an affine toric variety
$Y={\rm Spec}\,\CC[K^\dual\cap N^\dual]$. The hyperplane condition on the 
set $A$ implies that 
this variety has Gorenstein singularities. 
Moreover, any regular (i.e. projective) triangulation 
of the polytope $Q$ with vertices among the elements of $A$ induces 
a stacky fan structure $\SSigma$ supported on the cone $K.$ 
The notion of a regular triangulation will be explained below.
There is a natural crepant birational morphism 
$
\pi:X_\SSigma\to Y
$
from the induced smooth toric Deligne-Mumford stack $X_\SSigma$
to the toric affine Gorenstein singularity $Y.$ 

\medskip

A {\em triangulation} of the set $A= \{v_1, \ldots, v_n \}$ is a triangulation of the 
polytope $Q$ such that all
vertices are among the the points $v_1, \ldots, v_n.$ 
A marked polytope $(P,B)$ consists of the convex polytope $P$
and a finite subset $B$ of $P$ such that $P=\hbox{Conv}(B).$

\begin{definition}
\cite[Definition 7.2.1]{GKZbook}
\label{def:polysub}
Given the set $A \subset N = \ZZ^d,$
a {\it polyhedral subdivision} $\Pa$ of $(Q,A)$ is a family 
of $(d-1)$-dimensional marked polytopes  
$(Q_i,A_i), i=1, \ldots, l,$ $A_i \subset A,$ such that
any intersection $Q_i \cap Q_j$ is a face (possible empty) of both 
$Q_i$ and $Q_j,$ 
$$
A_i \cap (Q_i \cap Q_j) = 
A_j \cap (Q_i \cap Q_j),
$$
and the union of all $Q_i$ coincides with $Q.$
\end{definition}

A function $f : Q \to \RR$ is said to be 
{\it $\Pa$--linear} for the polyhedral subdivision 
$\Pa=(Q_i,A_i)$ if $f$ is continuous and the 
restriction of $f$ to each $Q_i$ is affine-linear. 
A function $\psi : A \to \RR$ is said to be
{\it $\Pa$--linearizable}
if there exists a {\it $\Pa$--affine}
function $g_\psi : Q \to \RR$ such that 
$g_\psi (v)=\psi(v)$ for any $v \in \cup A_i.$  
A continuous function $g : Q \to \RR$ is 
{\it concave}, if for any $x,y \in Q,$ 
we have $g(tx + (1-t) y) \geq tg(x) + (1-t) g(y),$ $0 \leq t \leq 1.$ 
For any polyhedral subdivision 
$\Pa=(Q_i,A_i),$ let $C(\Pa)$ denote the cone 
of $\Pa$--linearizable functions 
$\psi : A \to \RR$ such that the associated 
function $g_\psi : Q \to \RR$ is concave 
and $g_\psi (v) \geq \psi(v)$ for all 
$v \in A \setminus \cup A_i.$ 
We say that the subdivision $\Pa$ is {\it regular}
if the cone $C(\Pa)$ has relative 
non-empty interior inside the linear 
space of $\Pa$--linearizable functions. 

\medskip

For two polyhedral subdivisions $\Pa=(Q_i,A_i)$ and $\Pa'=(Q'_j,A'_j),$ we 
shall say that $\Pa$ is a {\it refinement} $\Pa'$ if, for each $j,$ the collection
of $(Q_i,A_i)$ such that $ Q_i \subset Q'_j,$ forms
a polyhedral subdivision of $(Q'_j,A'_j).$ The set of polyhedral subdivisions is then 
naturally a partially ordered set (poset) with respect to refinement. 
Triangulations are minimal elements of this poset. The maximal element 
is the polyhedral subdivision $(Q=\hbox{conv}(A),A).$  

\medskip

The cones $C(\Ta)$ for all the regular triangulations of $(Q,A)$ together
with all the faces of these cones form a complete fan $F(A)$
in $\RR^A= \RR^n$ 
called the {\it secondary fan}. We choose a translation invariant 
volume form vol in $\RR^{n-1}$ (or in $\RR^n$) such that elementary simplex
in the lattice $N = \ZZ^n \subset \RR^n$ has volume $1.$ The characteristic
function of the triangulation $\Ta$ is the function $\phi_\Ta : A \to \RR$
defined by
$$
\phi_\Ta (v):=\sum_{v \in {\rm Vert}(\sigma)} {\rm vol} (\sigma),$$
where the summation is taken over all the maximal simplices of the triangulations 
$\Ta$ for which $v$ is vertex. We set $\phi (v)=0$ if $v$ is not a vertex
of simplex in $\Ta.$ The {\it secondary polytope} $S(A)$ is the convex
hull in $\RR^A=\RR^n$ of the functions $\phi_\Ta$ for all the triangulations
$\Ta$ of $(Q,A).$ 
Theorem 7.1.7 in \cite{GKZbook} shows that 
the secondary polytope is $n-d$ dimensional and its vertices 
are the characteristics functions $\phi_\Ta$ for all triangulations
$\Ta$ of $(Q,A).$ Moreover, the secondary fan $F(A)$ is the normal 
fan of the secondary polytope $S(A).$ We identify $\RR^A= \RR^n$ 
with its dual in the canonical way. 

\medskip

Consider two regular triangulations
such that the corresponding vertices in the secondary polytope 
are joined by an edge. 
The two triangulations differ by a 
{\it modification} along a circuit. 
A {\it circuit} in $A$ is a minimal dependent subset $ \{ v_i, i \in I\}$
with $I \subset \{1,\ldots, n \}.$ 
In particular any circuit determines an integral relation 
of the form
\begin{equation}\label{circuit}
\sum_{i \in I_+ } l_i v_i + \sum_{i \in I_-} l_i v_i =0,
\end{equation}
with $I= I_+ \cup I_-,$ 
where the two subsets $I_+:= \{ i: l_i > 0 \}$ and  $I_-:= \{ i: l_i < 0 \}$
are uniquely defined by the circuit up to replacing $I_+$ by $I_-$. 
We will assume that the relation (\ref{circuit}) is {\it primitive},
i.e. that the integers $l_i$ have no common prime factor. Given a 
circuit, 
one can write a (possibly) non-primitve relation 
by choosing
\begin{equation*}
|l_i|:= {\rm vol} ({\rm conv} (v_i, i \in I \setminus {i})). 
\end{equation*}
When there is no danger of confusion, we may call the index subset $I$ a circuit. 
The minimizing condition in the definition of a circuit implies that the 
sets conv($v_i, i \in I_+$) and conv($v_i, i \in I_-$) intersect in their common interior point. 
The polytope determined by $v_i,$ $i \in I,$
admits exactly two triangulations $\Ta_+(I)$
and $\Ta_-(I)$ defined by the simplices conv($\{v_j, j \in I \setminus i$), for 
$i \in I_+$, respectively $i \in I_-$.

\medskip

Suppose that the regular triangulations $\Ta$ and $\Ta^\prime$ 
of $A$ are obtained from each other by a modification 
along a circuit $I.$ We say that a 
subset $J \subset A \setminus I,$ is {\it separating}
for $\Ta$ and $\Ta^\prime$ if, for some $i \in I,$ the set
of $v_i,$ with $i \in (I \setminus i) \cup J$ is the set of vertices of a 
simplex (of maximal dimension) of $\Ta_1.$ It turns out 
that a separating set $J$ has the property that the sets
$(I \setminus i) \cup J,$ for $i \in I_+,$ determine simplices 
in $\Ta$ and 
the sets $(I \setminus i) \cup J,$ for $i \in I_-,$ 
determine simplices in $\Ta^{\prime}.$ 

\medskip

\begin{proposition}\cite[Prop 7.2.12]{GKZbook} \label{prop:subd}
The polyhedral subdivision
$\Fa=\Fa(\Ta, \Ta^\prime)$
corresponding to the edge of the secondary polytope
joining the vertices determined by $\Ta$ and $\Ta'$ consists
of the simplices $(conv(K),K)$ which $\Ta$ and $\Ta'$
have in common and the polyhedra $(conv(I \cup J), I\cup J)$
for all separating subsets $J \subset A \setminus I.$
\end{proposition}

\subsection{Principal $A$-Determinants}
Define $\nabla_{A}$ as the Zariski closure in $\CC^n$
of the set of polynomials $f=\sum_{1 \leq j \leq n} 
a_i x^{v_i}$ in $\CC[x_1, \ldots, x_d]$
such that there exists some
$y \in (\CC^{\times})^n$ with the property that
$f=0$ is singular at $y.$
By definition, the discriminant $\Delta_{A} \in \ZZ[a_1, \ldots, 
a_n]$ is the irreducible polynomial (defined up to a sign)
whose zero set is given by the union of the irreducible
codimension $1$ components of $\nabla_{A}.$ For the case 
codim$\nabla_{\cA} >1,$ one sets  $\nabla_{\cA}=1.$ 

\medskip

For any non-empty face $\Gamma$ of the polytope $Q,$
let $i(\Gamma)$ denote the index of the sublattice
$\ZZ( A \cap \Gamma)$ inside the lattice 
$N \cap \RR \Gamma,$
\begin{equation}
i(\Gamma):= [N \cap \RR \Gamma: \ZZ (A \cap \Gamma)].
\end{equation}
Recall that 
$S=\ZZ_{\geq 0} A$ is the semigroup generated by $A.$
Furthermore, let $S/\Gamma$ denote the image semigroup 
of $S$ in the quotient free group $N_\RR/ \RR \Gamma,$
with $N_\RR= N \otimes_\ZZ \RR.$ The semigroup 
$S/\Gamma$ generates a pointed cone in 
$N_\RR/ \RR \Gamma,$ so we can define 
\begin{equation}\label{def:u}
u(\Gamma):= {\rm vol} ( {\rm conv} \big(S/\Gamma) \setminus 
{\rm conv} (S/\Gamma \setminus \{0\}) \big),
\end{equation}
where the volume form on $N_\RR/ \RR \Gamma$
is induced from the standard volume form 
on $N.$ 

\begin{definition}\cite[Theorem 10.1.2]{GKZbook}\label{def:ea}
The {\it principal $A$-determinant} $E_A$ is the 
polynomial in $\ZZ[a_1, \ldots, a_n]$
defined as 
\begin{equation}
E_A:= \prod_{\Gamma} (\Delta_{A \cap \Gamma})^{u(\Gamma) \cdot i(\Gamma)}, 
\end{equation}
where the product is taken over all the non-empty faces of the
polytope $Q={\rm conv} (A).$
\end{definition}

Note that we have taken the liberty of not using the definition 
given in \cite{GKZbook}. Of course, the above definition 
is equivalent to the one given in that book. For the faces $\Gamma$ such that 
codim$\nabla_{\cA} >1,$ the corresponding factor 
in the product is equal to $1.$ This implies that
the product is taken over the vertices and the faces $\Gamma$
such that $A \cap \Gamma$ is not simplicial and there 
is no proper subface $\Gamma^{\prime}$ of $\Gamma$ 
such that the linear 
relations among the elements of $A \cap \Gamma$ are 
linear relations among the elements of $A \cap \Gamma^{\prime}.$

\medskip

The principal $A$--determinant has the following remarkable property. 
\begin{theorem}\cite[Theorem 10.1.4]{GKZbook}
For a given set $A,$ the Newton polytope of $E_A$ coincides with
the secondary polytope $S(A).$ In particular, the vertices of the 
Newton polytope of $E_A$ are given by the characteristic functions 
$\phi_\Ta$ for all the regular triangulations $\Ta$ of $(Q,A).$
\end{theorem}

\section{K-theory and discriminant intersection multiplicities}
\label{sec:k}

We will now compute the dimensions of the Grothendieck rings 
of some toric stacks involved in the conjecture discussed 
in this paper. 
Let $\Gamma$ be a non-empty face of the polytope $Q={\rm conv}(A).$ 

\begin{definition}\label{def:fangamma}
The stacky fan $\SSigma_\Gamma=
(\Sigma_\Gamma, \{ \pi (v_i) \}_{i \in {\mathcal I}}),$ with
$\pi : N \to N / \ZZ(A \cap \Gamma)$ the canonical map, is defined by 
a simplicial fan
$\Sigma_\Gamma$ in the real linear space 
$(N / \ZZ(A \cap \Gamma) \otimes \RR= N_\RR/\RR\Gamma$
with rays generated by the images of the 
vectors $v_i$ for $i \in \Ia.$
The set $\mathcal I$ consists of all indices $i$
such that the image of $v_i$ in $N_\RR/\RR\Gamma$
is contained in the closure of the set
$$
{\rm conv} \big(S/\Gamma) \setminus 
{\rm conv} (S/\Gamma \setminus \{0\}),
$$
where $S/\Gamma$ is the image semigroup 
$S=\ZZ A$ in $N_\RR/ \RR \Gamma.$
\end{definition}

It is important to note that the set $A$ 
and the nonempty face $\Gamma$ uniquely determine 
the vectors $v_i, i \in \Ia,$ but the cones 
of the simplicial fan $\Sigma_\Gamma$ are not uniquely determined. 
However, the following statement holds. 

\begin{proposition}
For any two choices of stacky fans $\SSigma_\Gamma$ 
and $\SSigma_\Gamma^\prime$ as above,
the bounded derived categories of coherent categories $D^b(X_{\SSigma_\Gamma})$ 
and $D^b(X_{\SSigma_\Gamma^\prime})$ are equivalent.  
\end{proposition}

\begin{proof} The stacky fans $\SSigma_\Gamma$ 
and $\SSigma_\Gamma^\prime$ are connected by a finite sequence of 
toric birational flops. Each of them induces an equivalence of 
the bounded derived categories of coherent sheaves 
on the corresponding DM stacks \cite{BO}, \cite{Kaw}, so the result follows.
\end{proof}

Moreover, for any choice of the stacky fan $\SSigma_\Gamma,$
we can prove the  
following result.

\begin{proposition}\label{prop:kgamma}
The rank of the Grothendieck group
of the DM stack $X_{\SSigma_\Gamma}$ 
is equal to the product $u(\Gamma) \cdot i(\Gamma).$
\end{proposition}

\begin{proof}
Since $N$ is a lattice, 
the quotient group $N/ N \cap \RR \Gamma$ is torsion free.
This implies that the torsion 
part of the quotient $N / \ZZ(A \cap \Gamma)$ is isomorphic 
to the torsion part of the quotient 
$N \cap \RR \Gamma/ \ZZ (A \cap \Gamma).$ The order
of the torsion part is given by the index $i(\Gamma).$ 
The discussion in section 6 of \cite{BH-K}, as well as 
the work of Jiang-Tseng \cite{Jiang}, show that the rank of
$K_0(X_{\SSigma_\Gamma})$ is equal to $K_0(X_{\Sigma_\Gamma})
\cdot i(\Gamma).$

In order to finish the proof, we have to show that the dimension of
$K_0(X_{\Sigma_\Gamma})$
is equal to $u(\Gamma).$ Formula \ref{def:u}
implies that $u(\Gamma)$ is equal to the sum of the volumes of 
the (simplicial) maximal dimensional cones in the fan $\Sigma_\Gamma.$ 
The result follows from 
\cite[Proposition 3.20]{DKK} which shows that the desired rank 
is equal to the sum of the volumes of the maximal dimensional cones in the fan 
$\Sigma_\Gamma.$ 
\end{proof}

\medskip

We now consider two triangulations $\Ta$ and $\Ta^\prime$ 
corresponding to vertices of the secondary polytope 
which are joined by an edge $F$. According to Proposition \ref{prop:subd},
there exists a circuit $I=\{ v_i \}, i \in I,$ associated to this
edge. Let $\ZZ I: = \sum _{i \in I} \ZZ v_i$ be the sublattice 
generated in $N$ by the vectors of the circuit. Recall that 
the difference between the
simplices in $\Ta$ and $\Ta^\prime$ is determined 
by the separating sets $J \subset A \setminus I$ (cf. 
Proposition \ref{prop:subd}). Note that for 
any separating set $J,$ the cone determined by the 
polytope conv$(J)$ is simplicial. Let $K \subset A$ denote 
the subset of elements in $A$ that belong to some separating set
associated to the edge $F.$

\begin{definition}\label{def:fanz} 
The stacky fan $\SSigma_F=
(\Sigma_F, \{ \pi (v_i) \}_{i \in K}),$ with
$\pi : N \to N / \ZZ I$ the canonical map,
is defined by the fan
$\Sigma_F$ in the real linear space 
$N_\RR/ \RR I$ whose maximal cones
are images under the map $\pi$ 
of cones over 
conv$(J),$ for some separating set $J.$
 \end{definition}

Denote by 
$Z_F$ the toric DM stack defined by the stacky fan 
$\SSigma_F.$ 
For the general theory of spherical functors,
see \cite{AL}.

\begin{proposition}\label{prop:ez}
There exists a ``wall monodromy" spherical functor $D^b (Z_F) \to D^b (X)$
where $X$ is the toric DM stack induced by either one 
of the triangulations $\Ta$ or $\Ta^\prime.$
\end{proposition}

\begin{proof} The $EZ$-twist argument in \cite{EZ} is mostly formal 
and can be adapted to the toric DM stacky context. The result 
also follows from the work of Halpern-Leistner-Shipman \cite{HLS}. 
In the physics literature, the spherical twist induced 
by this spherical functor is known as the ``wall monodromy".
\end{proof}

\begin{proposition}\label{prop:ki}
The rank of the Grothendieck group
of the DM stack $Z_F$ 
is equal to the sum 
$$\sum_{J-{\rm sep}} [N: \ZZ (I \cup J)]$$
over all the separating sets $J.$
\end{proposition}

\begin{proof}
Since $I \cup J$ is a simplex, and 
$\ZZ I \cap \ZZ J= (0),$ we can write a non-canonical
splitting for the torsion group 
$N/ \ZZ(I \cup J)$ as a direct sum of finite torsion groups
$\ZZ^k/ \ZZ I \oplus \ZZ^{d-k}/\ZZ J.$ Note that 
the index $[\ZZ^{d-k}/\ZZ J]$ is the volume of the 
cone generated by the separating set $J$ in the 
fan $\Sigma_I$ in $N_\RR / \RR I$ associated to the 
toric DM stack $Z_F.$ As in the proof of Proposition \ref{prop:kgamma},
we conclude that the rank of the Grothendieck group 
of the toric DM stack is indeed the sum $\sum_J [N: \ZZ (I \cup J)]$
over all the separating sets $J.$
\end{proof}

The main result of this note is the following theorem. 

\begin{theorem} \label{thmm}
For any edge $F=[\phi_{\Ta}, \phi_{\Ta^\prime}]$ of the secondary polytope, 
between the triangulations $\Ta$ and $\Ta^\prime.$ 
the following equality holds:
\begin{equation}\label{formula:k0}
\rk (K_0 (D^b(Z_F))) = \sum_{\Gamma \subset Q} n_{\Gamma,F} \cdot \rk (K_0(D_\Gamma))), 
\end{equation}
for some non-negative integers $n_{\Gamma,F},$ and  
the summation is taken over all the non-empty
faces $\Gamma$ of the polytope $Q.$ 
\end{theorem}

The precise combinatorial definition of the non-negative 
integers $n_{\Gamma,F}$ 
will be given below in Definition \ref{def:n}. It
arises naturally as a by-product of the proof of the theorem
and it is close in spirit to the string theoretical analysis in 
\cite{AGM}. Geometrically,
the index $n_{\Gamma,F}$ coincides in most cases with 
the intersection 
multiplicity of the discriminant component $\nabla_{\Gamma,F}$
with the rational curve induced by edge $F$ in the 
DM toric stack defined by the secondary fan. This point of view 
has been discussed in \cite{HLS}, \cite{KS}.

\begin{proof}
We will prove the result by 
giving two interpretations to the coefficient 
restriction 
of the principal $A$-determinant $E_A$
to the edge $F$ of the secondary polytope. 
By \cite[Prop 6.1.3]{GKZbook},
the coefficient 
restriction of any Laurent polynomial
$P(x_1, \ldots, x_n)$ to any face $F$
of its Newton polytope $S(A)$
has the leading term of 
the polynomial in $t$ 
$$t \mapsto P(t^{\psi_1} x_1, 
\ldots , t^{\psi_n} x_n)$$
equal to 
$t^{\psi(F)} P_F (x_1, \ldots, x_n),$
where $\psi$ is a linear support function for the 
face $F,$ and $P_F$ is the coefficient 
restriction of $P$ to the face $F.$ Recall 
that $\psi : \RR^A \to \RR$ is a linear support function for the 
face $F$ if $F$ is the maximal face the secondary polytope 
$S(A)$ where $\psi$ attains its maximum value. 

Let $\Pa=(Q_i,A_i)$ be a polyhedral subdivision of the marked polytope $(Q,A).$
Theorem 10.1.12' in \cite{GKZbook} shows that, up to multiplication by a constant,
the coefficient
restriction of $E_A$ to the edge $F(\Pa)$ of the secondary polytope equals
$$
\prod_i (E_{A_i})^{[N: \ZZ A_i]} \ .
$$
In particular, if $\Fa$ is the polyhedral subdivision corresponding to the 
circuit $I$ and the edge $[\phi_\Ta, \phi_{\Ta^\prime}]$ for 
two triangulations $\Ta,$ $\Ta^\prime$
(see Proposition \ref{prop:subd}), then, up to multiplication by a 
constant, $E_A$ is equal to
$$
\prod_{i \in K} (E_{A_i})^{[N: \ZZ A_i]} \ 
\cdot
\prod_{J-{\rm sep}} (E_{I \cup J})^{[N: \ZZ (I \cup J)]} \ ,
$$
where $({\rm conv}(A_i), A_i)_{i \in K}$ is the family of simplices
that $\Ta$ and $\Ta^\prime$ have in common, 
and the second product is taken over all the separating sets $J$ 
such that $J \cup I$ has maximal dimension. Note that 
each such $J \cup I$ contains as a unique circuit the circuit $I.$ 
In the terminology of \cite[pg 309]{GKZbook}, the set
$J \cup I$ is {\it weakly dependent}. As such, the argument 
on \cite[pg 309]{GKZbook} shows that, up to a multiplication 
by a constant, $E_{I \cup J}$ is in fact equal to the 
the discriminant $\Delta_{I \cup J},$ which by \cite[Prop 9.1.8]{GKZbook}
is a non-zero scalar multiple of the polynomial
\begin{equation}\label{di}
\Delta_I:=(\prod_{i \in I_+} l_i^{l_i}) \prod_{i \in I_-} a_i^{-l_i}
-
(\prod_{i \in I_-} l_i^{-l_i}) \prod_{i \in I_+} a_i^{l_i},
\end{equation}
where $\sum_{i \in I_+ \cup I_-} l_i v_i=0$ 
is a primitive integer relation (unique up to sign) 
associated to the circuit $I.$ Note \cite[Prop 9.1.8]{GKZbook}
assumes that the circuit generates the ambient lattice so 
the choice
$
|l_i|:= {\rm vol} ({\rm conv} (v_i, i \in I \setminus {i})),
$
would work in that case. We do not make that assumption. 

\medskip

Consider a linear form $\psi :\RR^A \to \RR$
lying in the 
interior of cone $C(\Fa)$ (the normal cone
to the edge $F$) of the secondary fan as
defined in section \ref{section:A}.
This means that the edge $F$ is the 
supporting face of $\psi.$ In particular,
$\psi$ is equal to a constant $\psi(F)$ along $F,$ 
and 
$\psi(v_i) < \psi(F)$ when $v_i$ is not in $F.$
It follows that the coefficient of
the leading term in $t$ of the Laurent polynomial 
$t \mapsto E_A (\sum_i t^{\psi_i} a_i x^{v_i})$ is, up to
multiplication by a constant and a monomial in the variables $a_i,$ the product
\begin{equation}\label{dom}
\prod_{J-{\rm sep}} (E_{I \cup J})^{[N: \ZZ (I \cup J)]} 
=(\Delta_I)^{\sum_J [N: \ZZ (I \cup J)]}
\end{equation}
where the sum and product in the formula above are over all the 
separating sets $J,$ a the discriminant polynomial 
$\Delta_I$ is defined by (\ref{di}).

\medskip

On the other hand, according to Definition \ref{def:ea}
the principal $A$--determinant $E_A$ is the product 
$$
E_A:= \prod_{\Gamma} (\Delta_{A \cap \Gamma})^{u(\Gamma) \cdot i(\Gamma)}, 
$$
where the product is taken over all the non-empty faces of the
polytope $Q={\rm conv} (A).$ Due to the properties of Minkovski sums for Newton polytopes,
the linear form $\psi$ attains its maximum value 
on the Newton polytope of $\Delta_{A \cap \Gamma}$ at either a vertex
or an edge parallel to $F.$
This means that, up to
multiplication by a constant and a monomial in in the variables $a_i$, 
the product
the coefficient of the
leading term in $t$ of the Laurent polynomial 
$t \mapsto \Delta_{A \cap \Gamma} (\sum_i t^{\psi_i} a_i x^{v_i})$ is a power of
$ 
\Delta_I.
$

\begin{definition}\label{def:n}
The multiplicity $n_{\Gamma,F}$ is 
the non-negative integer defined by the property that, up to
multiplication by a constant and a monomial in the variables $a_i$, 
the coefficient of the
leading term in $t$ of the Laurent polynomial 
$t \mapsto \Delta_{A \cap \Gamma} (\sum_i t^{\psi_i} a_i x^{v_i})$ is equal
to $(\Delta_I)^{n_{\Gamma,F}}.$
\end{definition}
The theorem follows as a direct consequence \label{eq:n}
of the 
results of Propositions \ref{prop:kgamma} and \ref{prop:ki}.
\end{proof}

\begin{remark}
Note that $n_{\Gamma,F}=0,$
when $\Gamma$ is a vertex of $Q.$ 
Furthermore, the definition shows that, if the circuit $I$ determined by the  
edge $F$ is not contained in the face $\Gamma,$ then $n_{\Gamma,F}=0.$
\end{remark}

\begin{remark}
In practice, calculating the multiplicities 
$n_{\Gamma,F}$ can be computationally daunting.
We will provide some examples in the next section. 
On the other hand, this definition avoids the well known 
difficulties related to
the local topological picture of intersecting discriminants.
Moreover, our definition does not use the
Horn uniformization (see Kapranov \cite{KapHorn})
of the $A$--discriminant
which is a useful tool but only provides a birational parametrization
of the discriminant.
\end{remark}

The previous theorem adds to the body of evidence 
supporting the conjecture of 
Aspinwall-Plesser-Wang \cite{APW}. The importance 
of multiplicities and other aspects 
of this conjecture have been discussed by
Kite--Segal \cite{KS}. 

\begin{conjecture}\label{conj:apw}

\

$1)$ For any edge $F$ of the secondary polytope, the category $D^b(Z_F)$
admits a semiorthogonal decompositon consisting of 
$n_{\Gamma,F}$ components $D^b (X_{\SSigma_\Gamma})$ for each
non-empty face $\Gamma$
of the polytope $Q={\rm conv}(A).$ 

$2)$ For any non-empty face $\Gamma$ of the polytope $Q,$
there exists a spherical functor
$D^b(D_\Gamma) \to D^b (X)$ for any toric DM stack $X$ 
determined by a triangulation corresponding to a vertex of the secondary 
polytope. 
\end{conjecture}

\begin{remark}\label{rm1}
One of the hypotheses in the definition of an $A$-set is that 
the set $A$ generates the lattice $N=\ZZ^d.$ 
This hypothesis guarantees that the multiplicity of the ''primary"
component of the discriminant associated to the face 
$\Gamma=Q$ has $u(\Gamma)= i(\Gamma)=1$ and 
$Z_\Gamma$ is a point ${\rm Spec} \, \CC.$
This assumption is likely not essential. 
We could simply ask 
that $A$ linearly
generates the vector space $N_\RR=\RR^d.$ The adapted results of this note 
should continue to hold, although the multiplicities involved gain the 
extra factor $[N: \ZZ A].$ On the other hand, the technical details 
in \cite{GKZbook} are obtained for $A$-sets in this more restricted context, 
so one would have to work out and adapt the calculations to cover the more 
general situation. 
\end{remark}

\section{Applications}
\label{section:examples}

\example\label{ex1}
Recall that for a general set $A$ in the lattice $N=\ZZ^d,$  
the rational polyhedral cone $K$ generated by $A$ 
defines a Gorenstein affine toric variety
$Y={\rm Spec}\,\CC[K^\dual\cap N^\dual]$. 
For any stacky resolution 
$\pi:X_\SSigma\to Y$ induced by the a regular 
trangulation of the polytope $Q,$
the structure of the exceptional set is 
quite complicated. It may contain irreducible components
of different dimensions. Consider an edge 
of the secondary polytope and 
let $I=I_+ \cup I_-$ 
be the associated circuit in $A$. We assume 
that $I$ is maximal dimension, $\dim I = \dim C.$
Let $X_+$ and $X_-$ denote the cooresponding toric
stacky resolutions of $Y$ that differ by a flop.
In this case $Z_F$ is the stacky point $[\rm{Spec} \, \CC/\ZZ_n]$
where $n$ is the index of the sublattice generated 
by the circuit $I$ in $N, n=[N: \ZZ I].$
The sets $I_+$ and $I_-$ generate cones that determine 
the exceptional loci $E_+ \subset X_+,$ and $E_- \subset X_-,$
and, by Proposition \ref{prop:ez}, we have spherical
functors $D^b({\rm Spec} \, \CC/\ZZ_n]) \to D^b(E_\pm) \hookrightarrow D^b(X_\pm).$
The category $D^b([\rm{Spec} \, \CC/\ZZ_n])$ splits into $n$
orthogonal copies of $D^b (\rm{Spec} \, \CC),$ so we get the spherical
functors predicted by Kontsevich's conjecture described in 
the introduction. Theorem \ref{thmm} implies that
$n_{F,Q}=n,$ where $F$ is the edge of the secondary 
polytope associated to the flop. This multiplicity 
statement is not obvious if one attempts a direct calculation. 
In the transversal case $n_{F,Q}=1,$ the
analytic
monodromy calculations from \cite{hor}, \cite{BH2} 
for the classical GKZ $D$--module are valid, so we get that: 

\begin{corollary}\label{cor:tr}
The wall monodromy spherical functor induced 
by the edge $F$ is the spherical functor 
$D^b {(\rm Spec} \, \CC) \to D^b(E_F) \to D^b(X),$ 
where $E_F \subset X$  is the exceptional locus 
corresponding to the circuit $I.$ The $K$--theory 
action of this functor matches the analytical continuation
functor along a corresponding loop $\gamma_F$
(see \cite{BH2}).
\end{corollary}

A similar argument expands the range of cases 
where Kontsevich's conjecture is true beyond the 
the cases considered in  \cite{hor}. We hope to return 
to this issue in future work. 

\

\example \label{exa2}
Let $X$ be the resolution of the $A_3$ singularity, 
with $v_0=(1,0), v_1=(1,1), v_2=(1,2), v_3=(1,3), 
v_4=(1,4).$ As it is well known, the secondary polytope
is combinatorially equivalent to a cube in $\RR^3.$
The principal $A$--determinant is 
\begin{equation*} 
\begin{split}
E_A&= a_0 a_4 \ \Delta_Q= a_0 a_4
 (256 a_0^3 a_4^3-192 a_0^2 a_1 a_3 a_4^2-128 a_0^2 a_2^2 a_4^2\\&+144 a_0^2 a_2 a_3^2 a_4 
-27 a_0^2 a_4^4
+144 a_0 a_1^2 a_2 a_4^2-6 a_0 a_1^2 a_3^2 a_4-80 a_0 a_1 a_2^2 a_3 a_4\\
&+18 a_0 a_1 a_2 a_3^3
+16 a_0 a_2^4 a_4-4 a_0 a_2^3 a_3^2-27 a_1^4 a_4^2\\&+18 a_1^3 a_2 a_3 a_4 
-4 a_1^3 a_3^3-4 a_1^2 a_2^3 a_4+ a_1^2 a_2^2 a_3^2).
\end{split}
\end{equation*}

\medskip

It is clear that in this case $Z_Q$ is the point
${\rm Spec} \, \CC.$ Let $X$ denote the DM stack $[\CC^2/\ZZ_4]$ whose
stacky fan has one cone and the rays generated by the 
vectors $v_0, v_4.$ Consider the edge $F_1$ corresponding to the 
birational transformation $X \leftrightarrow X_1,$ where $X_1$ is the 
toric DM stack with cones determined by the pairs $v_0,v_1$ and $v_1,v_4.$
The associated polyhedral subdivision is $({\rm conv} \{0,4\}, \{ 0,1,4\}),$
and the circuit relation $I$ is $3v_0 - 4 v_1 + v_4=0.$ The discriminant
$\Delta_I$ is $256a_0^3 a_4 - 27a_1^4,$ and the leading term with respect 
to the edge $F_1$ in the quartic discriminant $\Delta_Q$ is 
$$
256 a_0^3 a_4^3 - 27 a_1^4 a_4^2= a_4^2 \cdot \Delta_I.$$
This means that $n_{Q,F_1}=1,$ which is consistent with the fact
that $Z_{F_1}= {\rm Spec} \, \CC.$ The associated spherical functor is 
$$D^b({\rm Spec} \, \CC) \to D^b ([\CC^2/\ZZ_4]).$$

\medskip

Let $F_2$ denote the edge corresponding to the birational transformation
$X \leftrightarrow X_2,$ where $X_2$ is the 
toric DM stack with cones determined by the pairs $v_0,v_2$ and $v_2,v_4.$
The associated polyhedral subdivision is $({\rm conv} \{0,4\}, \{ 0,2,4\}),$
and the circuit relation $I$ is $v_0 - 2 v_2 + v_4=0.$ The discriminant
$\Delta_I$ is $4a_0 a_4 - a_2^2,$ and the leading term with respect 
to the edge $F_2$ in the quartic discriminant $\Delta_Q$ is 
$$
256a_0^3 a_4^3 
-128 a_0^2 a_2^2 a_4^2 + 16 a_0 a_2^4 a_4
= 16 a_0 a_4 \cdot \Delta_I^2.$$
This means that $n_{Q,F_2}=2$ which is consistent with the fact
that $Z_{F_2}$ is the stacky point $[{\rm Spec} \, \CC/\ZZ_2].$ The derived category 
$D^b([{\rm Spec} \, \CC/\ZZ_2])$ splits into two copies of 
$D^b({\rm Spec} \, \CC)$ and the associated 
wall monodromy spherical functor is 
$$D^b([{\rm Spec} \, \CC/\ZZ_2])
\to D^b ([\CC^2/\ZZ_4]).$$

\medskip

\example \label{ex3}
Let $X_1$ be the well known quasi--projective Calabi-Yau 
toric variety defined as the total space of the canonical bundle of $\PP^2.$
The toric structure is given by the vectors $v_0=(0,0,1), v_1=(1,0,1), 
v_2=(0,1,1)$ and $v_3=(-1,-1,1).$ The lattice of linear relations is one dimensional and 
generated by the vector $(-3,1,1,1),$ and the secondary polytope 
is the segment $F=[(3,2,2,2),(0,3,3,3)].$  
There are two toric birational models in this case, namely
$X_1$ and the stacky quotient $X_2 = [\CC^3/\ZZ^3].$ 
The principal $A$--determinant is $E_A= a_1^2 a_2^2 a_3^2 (a_0^3 -27 a_1 a_2 a_3).$
The multiplicity of the discriminant components corresponding 
to $v_1,v_2,v_3$ are all equal to $2,$ and $n_{F,Q}=1.$ The spherical
twist of $D^b(X_1)$ is determined by the spherical object $\Oa_{\PP2} \subset X_1.$

\example \label{ex4}
Let's consider the case the Calabi-Yau 
toric variety $X_1$ is the total space of the canonical bundle of the Hirzebruch surface $F_2.$ 
The toric structure is given by the vectors $v_0=(0,0,1), v_1=(1,0,1), 
v_2=(0,1,1), v_3=(-1,2,1)$ and $v_3=(0,-1,1)$ with the obvious cones.
The lattice of relations is two dimensional and 
generated by the vectors $(-2,0,1,0,1)$ and $(0,1,-2,1,0).$ There are four stacky 
toric birational models for $X_1.$ 

\medskip

\begin{figure}
\begin{center}
\includegraphics[height=6cm,width=8cm]{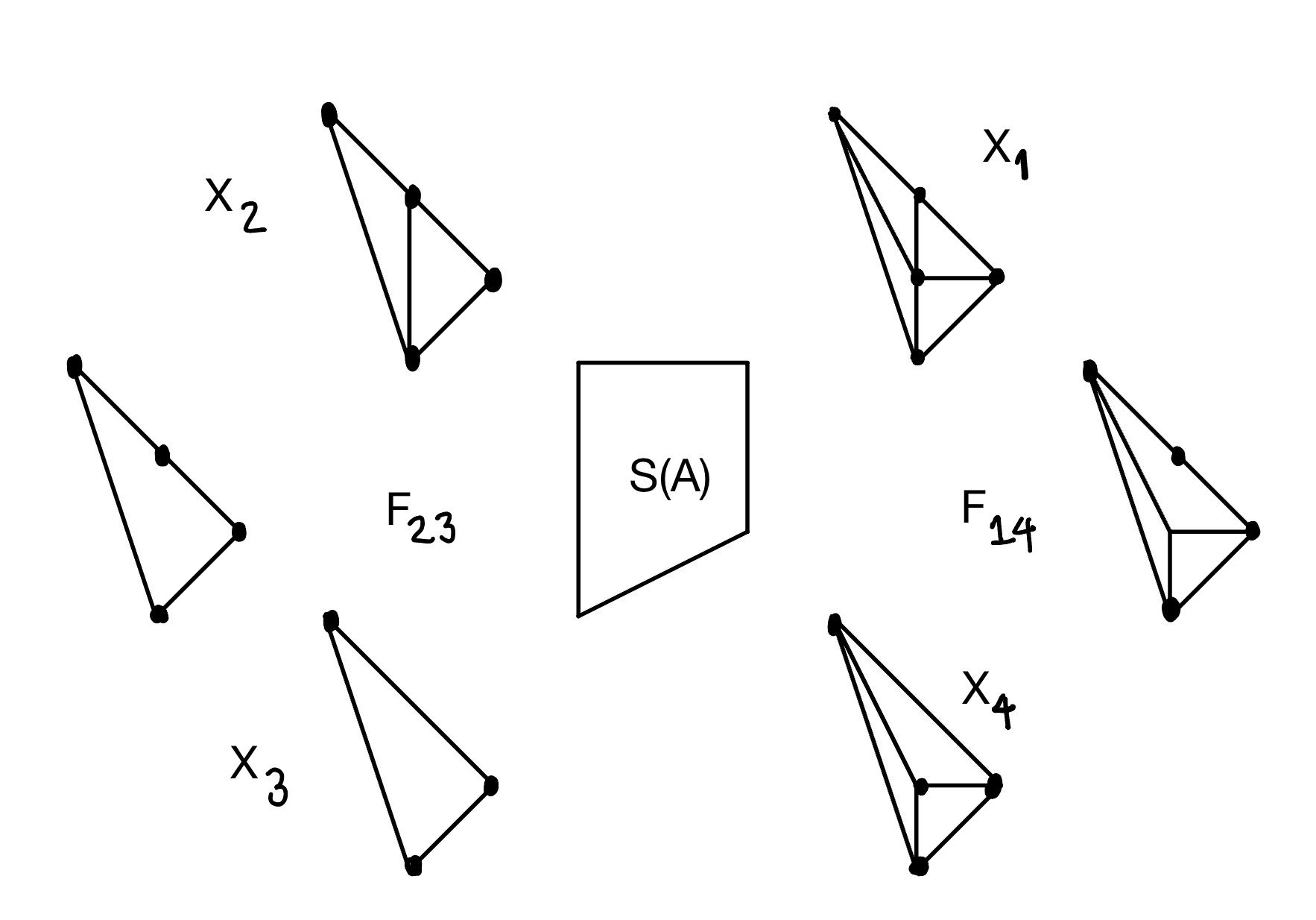}
\caption{The $F_2$ example.}
\label{fig:f2}
\end{center}
\end{figure}

The discriminant picture is obtained by studying the singularities 
of the Laurent polynomial $\sum a_i x^{v_i}.$ Besides the vertices
corresponding to $v_1, v_3$ and $v_4$ there are two other 
faces that contain circuits. The discriminant component 
determined by the face $\Gamma$ generated by the 
vectors $v_1, v_2$ and $v_3$ is 
$$
f_{\Gamma}= a_2^2 - 4 a_1 a_3.$$
For the maximal dimensional face $Q$ we obtain the discriminant
$$
f_{Q}= a_0^4 -8a_0^2 a_2 a_4 +16  a_2^2 a_4^2 -64 a_1 a_3 a_4^2.$$
The principal $A$-determinant is 
$$E_A=a_1^2 a_2^2 a_4^2 \cdot f_\Gamma \cdot f_Q.$$
If we choose the coordinates adapted to the large complex
structure point 
$$
x:= \frac{a_2 a_4}{a_0^2}, \quad y= \frac{a_1 a_3}{a_2^2}
$$
the we obtain that the two relevant discriminant components are given by 
$$y=\frac{1}{4} (1-\frac{1}{4x})^2 \quad \text{and} \quad 
y= \frac{1}{4}.$$
Figure 1 is suggestive and intuitively very useful and 
appeared for the first time in the paper by Dave Morrison \cite{DRM} on the $3$-fold octic
in $\PP^4(1,1,1,2,2)$ which is a higher dimensional generalization 
of our example.

\medskip

The analysis of the neighborhood of the point $A$ is completely analogous to the 
work \cite{AHK}, so we will not discuss here. 
By analzying the moduli space picture, 
we see that a rational curve $x={\rm constant}$ infinitesimally close to the 
$y$-axis corresponds to a transition from $X_1$,
the total space of the canonical bundle of $F_2$ 
to the stacky version $X_4$ of the weighted projective space 
$\PP(2,1,1).$ This curve intersects $E_A=0$ in the neighborhood 
of the point $B.$ 
The circuit $I$ associated to the 
corresponding edge of the secondary quadrilateral is 
$v_1 -2v_2 + v_3=0.$ One of the associated EZ-spherical functor 
is induced by the diagram 
\begin{equation*}
\begin{split}
\xymatrix{
E={\mathbb A}^1 \times {\mathbb P}^1 \; \ar@{^{(}->}[r] \ar[d]_q & X_1\\
Z={\mathbb A}^1&
}
\end{split}
\end{equation*}

\begin{figure}
\begin{center}
\includegraphics[height=6cm,width=8cm]{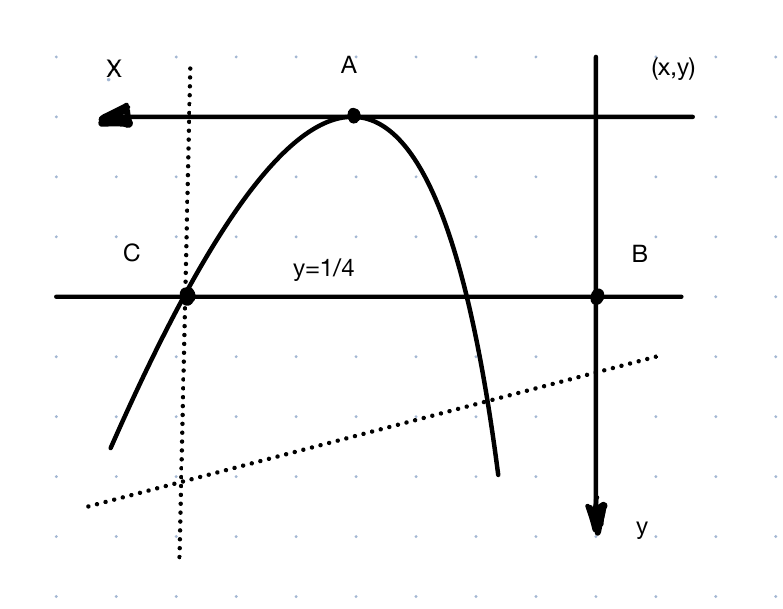}
\caption{The ''mirror" complex discriminant moduli space.}
\label{fig:2par}
\end{center}
\end{figure}

The birational transformation 
corresponding a rational curve $x=k$ 
($k$ a very large constant)
corresponds to a transition from $X_2$ to $X_3$ which
is the stacky quotient $[\CC^3 / \ZZ_4].$ 
This curve intersects $E_A=0$ in the neighborhood 
of the point $C.$ 
One of the 
associated spherical functors is 
is induced by the diagram 
\begin{equation*}
\begin{split}
\xymatrix{
E=[{\mathbb A}^1/\ZZ_2] \times {\mathbb P}^1 \; \ar@{^{(}->}[r] \ar[d]_q & X_4=[\CC^3 / \ZZ_4]\\
Z=[{\mathbb A}^1 / \ZZ_2]&
}
\end{split}
\end{equation*}
It is interesting to note that 
the circuit $I$ associated to this transition and determined by
corresponding edge of the secondary polytope is also 
$v_1 -2v_2 + v_3=0.$ 

\medskip

Recall that in determining multiplicities, the circuit discriminant
$\Delta_I= 4a_1 a_3 - a_2^2$ played a crucial role. For the transition
$X_1 \leftrightarrow X_4,$ we can consider the linear form $\psi$ 
in the cone $\Ca(\Fa_{14})$ whose
supporting face is the edge $F_{14}$ given by 
$$
\psi(v_0)= u, \psi(v_1)=  \psi(v_2)=  \psi(v_3)= 1, \psi(v_4)=0,
$$
with $u > 1/2$. A similar linear form $\psi$ will work for the transition
$X_2 \leftrightarrow X_3,$ with $u < 1/2$ in that case. The monomials in $f_Q$
$a_0^4, a_0^2 a_2 a_4, a_2^2 a_4^2, a_1 a_3 a_4^2$ have the weights 
$4u, 2u+1, 2, 2,$ respectively. The leading term for the edge $F_{14}$
is $a_0^4.$ The highest weight terms for the edge $F_{23}$ are 
$a_2^2 a_4^2$ and $a_1 a_3 a_4^2,$ so in this case the leading term 
for $f_Q$ is 
$$
16a_2^2 a_4^2 -64 a_1 a_3 a_4^2= - 16 a_4^2 \cdot \Delta_I.$$
We conclude that $n_{Q,F_{1,4}} =0$ and $n_{Q,F_{2,3}}=1,$ as 
Figure \ref{fig:2par} indicates. In this simple example, we do know
the true parametrizations of the components, so the weight calculation is not really needed.
However, in the general case, only the birational Horn parametrizations 
are available, so the weight calculation becomes necessary.

\medskip

\section{Final comments and some speculations}

The proposed constructions of spherical functors appearing 
in this paper seem to be related of the notion 
of co-sheaf over the topological 
space of the cone $K = \sum \RR v_i,$ in the 
sense of Bressler-Lunts \cite{BL},
section 6.7.1.

\

Since the defined spherical 
functors can be thought as 
categorical maps corresponding 
to the Bressler-Lunts strata, it is 
an intriguing question whether
the whole structure can be expressed
in the language of schobers introduced 
by Kapranov and Schechtman \cite{KaSc}.
This abstract realization would be consistent 
with the fact that the proposed $EZ$--spherical functors 
map ``stalks'' between various strata of 
a Whitney stratification of the characteristic
cycle of the GKZ hypergeometric system. 
Moreover, we expect that there are
analogues of the wall monodromy functors 
corresponding to all the toric strata 
in the stacky moduli space compactification 
given by the secondary fan, together with 
the corresponding semi-orthogonal 
decompositions with the categories $D_\Gamma$ 
as components.
Such a general picture would require a refinement
of our Minkovski sum argument in the proof of our multiplicity result. 
The stalks
in this case would be bounded derived categories 
of coherent sheaves of the form $D_\Gamma,$ 
indexed by the faces $\Gamma.$ In particular, 
if $\Gamma= \emptyset,$ the ``generic stalk'' is 
the bounded derived category of coherent sheaves on the 
ambient toric DM stack. 

\

Certain wall crossing phenomena were studied in the language of spherical pairs and 
schobers by Donovan \cite{Do} and, from a mirror symmetric 
point of view, by Nadler \cite{N}.
In a combinatorial context similar to ours,
\v{S}penko and Van den Bergh \cite{SvB}
produced a categorification of the GKZ system in the quasi-symmetric toric case. 
More work is needed in order to relate these points of view and make our speculation  
rigorous.

\end{document}